\documentclass[oneside,english]{amsart}
\usepackage{lmodern}

\usepackage[T1]{fontenc}
\usepackage[latin9]{inputenc}
\usepackage{amsthm}
\usepackage{amstext}
\usepackage{amssymb}
\usepackage{esint}

\makeatletter
\numberwithin{equation}{section}
\numberwithin{figure}{section}
\theoremstyle{plain}
\newtheorem{thm}{\protect\theoremname}
  \theoremstyle{plain}
  \newtheorem{prop}[thm]{\protect\propositionname}
  \theoremstyle{plain}
  \newtheorem{cor}[thm]{\protect\corollaryname}
  \theoremstyle{plain}
  \newtheorem{lem}[thm]{\protect\lemmaname}
  \theoremstyle{remark}
  \newtheorem{rem}[thm]{\protect\remarkname}
  \theoremstyle{definition}
  \newtheorem{example}[thm]{\protect\examplename}

\makeatother

\usepackage{babel}
  \providecommand{\corollaryname}{Corollary}
  \providecommand{\examplename}{Example}
  \providecommand{\lemmaname}{Lemma}
  \providecommand{\propositionname}{Proposition}
  \providecommand{\remarkname}{Remark}
\providecommand{\theoremname}{Theorem}


\begin{document}

\title[Spectral Properties]{Some spectral properties of pseudo-differential operators on the Sierpinski Gasket}

\begin{abstract} We prove versions of the strong Sz\"ego limit theorem for certain classes of pseudodifferential operators defined on the Sierpi\'nski gasket. Our results used in a fundamental way the existence of localized eigenfunctions for the Laplacian on this fractal. 

\end{abstract}

\date{\today}
\author{Marius Ionescu} 
\address{Department of Mathematics,  United States Naval Academy, Annapolis,
  MD, 21402-5002, USA}

\email{felijohn@gmail.com}
\thanks{This work was partially supported by a grant from the Simons Foundation (\#209277 to Marius Ionescu).}
\thanks{The first author would like to thank Kasso Okoudjou and the Department
  of Mathematics at the University of Maryland, College Park, and the
  Norbert-Wiener Center for Harmonic Analysis and Applications for
  their hospitality.}

 \author{Kasso~A. Okoudjou}
\address{Department of Mathematics, University of Maryland, College
  Park, MD 20742-4015, USA}
\email{kasso@math.umd.edu}

\author{Luke~G. Rogers}
\address{Department of Mathematics, University of Connecticut, Storrs,
  CT 06269-3009, USA}
\email{rogers@math.uconn.edu}
\maketitle

\section{Introduction}
The results of this paper have their origin in a celebrated theorem of
Szeg\"{o}, who proved that if $P_{n}$ is projection onto the span of
$\{e^{im\theta},0\leq m\leq n\}$ in $L^{2}$ of the unit circle and
$[f]$ is multiplication by a  positive $C^{1+\alpha}$  function for
$\alpha>0$ then $(n+1)^{-1}\log\det P_{n}[f]P_{n}$ converges to
$\int_{0}^{2\pi} \log f(\theta)\,d\theta/2\pi$.  Equivalently,
$(n+1)^{-1} \operatorname{Trace} \log P_{n}[f]P_{n}$ has the same
limit.  The literature which has grown from this result is vast,
see~\cite{Simon}; here we are interested in the generalizations that
replace $[f]$ with a pseudodifferential operator defined on fractal sets, in which case the
Szeg\"{o} limit theorem may be viewed as a way to obtain asymptotic
behavior of the operator using only its symbol. In the setting of
Riemannian manifolds there are important results of this latter type  due to Widom
\cite{Wid_JFA79} and Guillemin \cite{Gui_AnnMathSt79}.

In~\cite{OkRoStr_JFAA10} a standard generalization of the classical
Szeg\"{o} theorem was proved in the setting where the underlying space
is the Sierpi\'{n}ski Gasket and $P_{\Lambda}$ is the projection onto the
eigenspace with eigenvalues less than $\Lambda$ of the Laplacian
obtained using analysis on fractals.  In this situation the projection
is not Toeplitz, so most classical techniques fail and the proofs
in~\cite{OkRoStr_JFAA10} rely instead on the fact that most
eigenfunctions are localized.  The present work continues this
development to consider the case where $[f]$ is replaced by a
pseudodifferential operator.  Our results are of three related types.
Our main result, Theorem~\ref{thm:Mainthm}, gives asymptotics when one
considers $\operatorname{Trace} F( P_{n}p(x,-\Delta)P_{n})$ where
$p(x,-\Delta)$ is the pseudodifferential operator and $F$ is
continuous on an interval containing the spectrum of
$P_{n}p(x,-\Delta)P_{n}$.  Theorem~\ref{thm:GeneralSzego}, which is in
some sense a special case of the main result but for which we give a
different proof, considers the classical case of $\log\det
P_{n}p(x,-\Delta)P_{n}$.
Theorem~\ref{thm:specclusters} gives the asymptotics of clusters of
eigenvalues of a pseudodifferential operator in terms of the symbol.
   In all cases the proofs rely on the
predominance of localized eigenfunctions in the Laplacian spectrum on
the Sierpi\'{n}ski Gasket.

\section{\label{sec:Background}Background}

\subsection{Analysis on the Sierpi\'nski gasket}\label{ssec:analysisonSG}
The Sierpi\'{n}ski gasket $X$ is the unique non-empty compact fixed
set of the iterated function system
$\{F_{j}=\frac{1}{2}(x-a_{j})+a_{j}\}$, $j=1,2,3$, 
where $\{a_{j}\}$ are not co-linear in $\mathbb{R}^{2}$.  For
$w=w_{1}\dotsm w_{N}$ a word of length $N$ with all
$w_{j}\in\{1,2,3\}$ let $F_{w}=F_{w_{1}}\circ\dotsm\circ F_{w_{N}}$
and call $F_{w}(X)$ an $N$-cell. We will sometimes write the
decomposition into $N$-cells as 
$X=\cup_{i=1}^{3^{N}}C_{i}$ where each $C_{i}=F_{w}(X)$ for a word $w$ of
length $N$. 

Equip $X$ with the unique probability measure $\mu$ for which an
$N$-cell has measure $3^{-N}$ and the symmetric self-similar
resistance form $\mathcal{E}$ in the sense of Kigami~\cite{Kig_CUP01}.
The latter is a Dirichlet form on $L^{2}(\mu)$ with domain
$\mathcal{F}\subset C(X)$ and by the theory of such forms
(see~\cite{FOT}) there is a negative definite self-adjoint Laplacian
$\Delta$ defined by $\mathcal{E}(u,v)=\int (-\Delta u)v\,d\mu$ for all
$v\in\mathcal{F}$ such that $v(a_{j})=0$ for $j=1,2,3$.  This
Laplacian is often called the Dirichlet Laplacian to distinguish it
from the Neumann Laplacian which is defined in the same way but
omitting the condition $v(a_{j})=0$. 

A complete description of the spectrum of $\Delta$ was given in~\cite{FukShi_92}
using the spectral decimation formula of~\cite{rammal1983random}.
The eigenfunctions of $\Delta$ are described in~\cite{DalStrVin_FDESG_JFAA99}
(see \cite{Tep_JFA98,OkRoStr_JFAA10,OkStr_PAMS07,BaKi_JLMS97} or
\cite{Str_Prin06} for proofs).   Without going into details, the features needed for our work are as follows.
The spectrum  is discrete and decomposes naturally into three sets
called the $2$-series, $5$-series and $6$-series eigenvalues, which
further decompose according to the \emph{generation of birth}, which
is a number $j\in\mathbb{N}$.  All $2$-series eigenvalues have
$j=1$ and multiplicity $1$, each $j\in\mathbb{N}$ occurs in the
$5$-series and the corresponding eigenspace has multiplicity
$(3^{j-1}+3)/2$, and each $j\geq2$ occurs in the $6$-series with
multiplicity $(3^{j}-3)/2$.  Moreover the $5$ and $6$-series
eigenvalues are localized.  Suppose $j>N\ge1$ and
$\{C_{i}\}_{i=1}^{3^{N}}$ is the $N$-cell decomposition as above.   If
$\lambda$ is a $5$-series eigenvalue with generation of birth $j$ and
eigenspace $E_{j}$ then there is a basis for $E_{j}$ in which
$(3^{j-1}-3^{N})/2$ eigenfunctions are localized at level $N$, meaning
that they are each supported on a single $N$-cell, and the remaining
$(3^{N}-3)/2$ are not localized at level $N$; the number localized on
a 
single $N$-cell is $(3^{j-N-1}-1)/2$.  If $\lambda$ is a $6$-series
eigenvalue with generation of birth $j$ and eigenspace $E_{j}$ then
there is a basis for $E_{j}$ in which $(3^{j}-3^{N+1})/2$ of the
eigenfunctions are localized at level $N$, the 
remaining $(3^{N+1}-3)/2$ eigenfunctions are not localized at level
$N$, and the number supported on a single $N$-cell is $(3^{j-N}-3)/2$  (see, for example~\cite{OkRoStr_JFAA10}).

% There is a constant $c>0$ such that, for each $m\ge1$, the bottom
% $(3^{m+1}-3)/2$ eigenvalues of the Dirichlet spectrum of $-\Delta$
% are bounded by $c5^{m}$. Moreover, for each generation of birth $j$
% such that $1\le j\le m$, there are $2^{m-j}$ $5$-series and $2^{m-j-1}$
% $6$-series such eigenvalues (see, for example,  \cite[Proposition 1]{OkRoStr_JFAA10}).

%%%%%%%%%%%%%%%%%%

\subsection{Pseudo-differential operators on the Sierpi\'nski gasket}

We recall some of the theory developed in~\cite{IoRoStr_PSDO_2011} and use it to define
pseudo-differential operators on the Sierpi\'nski gasket.  Let $\Delta$ be
the Dirichlet Laplacian. Then the
spectrum $\operatorname{sp}(-\Delta)$ of $-\Delta$ consists of finite-multiplicity 
eigenvalues that accumulate only at $\infty$ (\cite{Kig_CUP01,Str_Prin06}).  Arranging them as
$\operatorname{sp}(-\Delta)=\{\lambda_{1}\le\lambda_{2}\le\dots\le\lambda_{n}\le\dots\}$
with $\lim_{n\to\infty}\lambda_{n}=\infty$, let $\{\varphi_{n}\}_{n\in\mathbb{N}}$
be an orthonormal basis of $L^{2}(\mu)$ such that $\varphi_{n}$
is an eigenfunction with eigenvalue $\lambda_{n}$ for all $n\ge1$. The
set $D$ of finite linear combinations of $\varphi_{n}$ is dense in
$L^{2}(\mu)$. 

If $p:(0,\infty)\to\mathbb{C}$ is a measurable then
\[
p(-\Delta)u=\sum_{n}p(\lambda_{n})\langle u,\varphi_{n}\rangle\varphi_{n}
\]
for $u\in D$ gives a densely defined operator on $L^{2}(\mu)$
called a \emph{constant coefficient pseudo-differential operator}.  If $p$ is bounded, then $p(-\Delta)$ extends
to a bounded linear operator on $L^{2}(\mu)$ by the spectral theorem.

If $p$ is a $0$-symbol in the sense of 
\cite[Definition 3.1]{IoRoStr_PSDO_2011} then $p(-\Delta)$ is a pseudo-differential
operator of order $0$.  Moreover, it is a singular
integral operator on $L^{2}(\mu)$ (\cite[Theorem
3.6]{IoRoStr_PSDO_2011}) and therefore extends to a bounded operator
on $L^q(\mu)$ for all $q\in (1,\infty)$.

If $p:X\times(0,\infty)\to\mathbb{C}$ is measurable
we define a \emph{variable coefficient pseudo-differential operator}
$p(x,-\Delta)$ as in \cite[Definition 9.2]{IoRoStr_PSDO_2011}:
\begin{equation}
p(x,-\Delta)u(x)=\sum_{n}\int_{X}p(x,\lambda_{n})P_{\lambda_{n}}(x,y)u(y)d\mu(y)\label{eq:psdo}
\end{equation}
for $u\in D$, where $\{P_{\lambda_{n}}\}_{n\in\mathbb{N}}$ is
the spectral resolution of $-\Delta$. If $p$ is a 0-symbol in the
sense of \cite[Definition 9.1]{IoRoStr_PSDO_2011}, then $p(x,-\Delta)$
extends to a bounded linear operator on $L^{q}(\mu)$ for all $q\in
(1,\infty)$ (\cite[Theorem 9.3 and Theorem 9.6]{IoRoStr_PSDO_2011}).
If $p:X\to\mathbb{R}_{+}$ is bounded and \emph{independent} of $\lambda$,
then the operator that $p$ determines on $L^{2}(\mu)$ is multiplication
by $p(x)$. In this case, we write $[p]$ for this operator, following the notation in \cite{OkRoStr_JFAA10}. 

We describe next the relationship between the spectrum of a constant
coefficient pseudo-differential operator $p(-\Delta)$ and $\operatorname{sp}(-\Delta)$
for a continuous map $p$. Of course, the result is valid for all
the spaces considered in \cite{IoRoStr_PSDO_2011}, not only the Sierpi\'nski
gasket.
\begin{prop}
\label{prop:spectrum_psdo}If  $p:(0,\infty)\to\mathbb{C}$
is continuous then $\operatorname{sp}p(-\Delta)=\overline{p(\operatorname{sp}(-\Delta))}$.\end{prop}
\begin{proof}
Let $\lambda$ be in the resolvent $\rho(p(-\Delta))$ of $p(-\Delta)$.
Then $\lambda I-p(-\Delta)$ has bounded inverse, so there
is $C>0$ such that  $\Vert(\lambda I-p(-\Delta))^{-1}u\Vert\le C\Vert u\Vert$
for all $u\in L^{2}(\mu)$. Let $v\in\operatorname{dom}p(-\Delta)$.
Then $(\lambda I-p(-\Delta))v\in L^{2}(\mu)$ and $\Vert(\lambda I-p(-\Delta))v\Vert\ge\frac{1}{C}\Vert v\Vert$.
 In particular, if $v=\varphi_{n}$ is an element of our $L^{2}(\mu)$ basis of eigenfunctions
of $-\Delta$ we obtain $\Vert(\lambda I-p(-\Delta))\varphi_{n}\Vert\ge1/C$. Since $p(-\Delta)\varphi_{n}=p(\lambda_{n})\varphi_{n}$
and $\Vert\varphi_{n}\Vert=1$ for all $n\in\mathbb{N}$, it follows
that $\lambda\notin\overline{p(\operatorname{sp}(-\Delta))}$.

For the converse, suppose $z\notin\overline{p(\operatorname{sp}(-\Delta))}$.
Then there is $K>0$ such that $\vert p(\lambda)-z\vert\ge K>0$ for
all $\lambda\in\operatorname{sp}(-\Delta)$ and thus $\vert p(\lambda)-z\vert^{-1}\le K^{-1}$
for all $\lambda\in\operatorname{sp}(-\Delta)$. Therefore $(p(-\Delta)-z)^{-1}$
is bounded on $L^{2}(\mu)$ and  $z\in\rho(p(-\Delta))$.
\end{proof}
%We will use the following corollary in Section \ref{sec:asymptotics-of-eigenvalue}.
\begin{cor}
\label{cor:spectrumifpinfty}If $p:(0,\infty)\to\mathbb{R}$ is 
continuous and $\lim_{\lambda\to\infty}p(\lambda)=\infty$
then $\operatorname{sp}p(-\Delta)=p(\operatorname{sp}(-\Delta))$.\end{cor}
\begin{proof}
The hypothesis implies that the only accumulation point for $\{p(\lambda_{n})\}_{n\in\mathbb{N}}$
is $\infty$ so we can apply  Proposition \ref{prop:spectrum_psdo}.
\end{proof}

\section{Szeg\"{o} Limit Theorems for pseudo-differential operators on the
Sierpi\'nski gasket\label{sec:Szeg-Limit-Theorems}}

Let $X$, $\mu$ and $\Delta$  be Section \ref{sec:Background}.
We also follow the notation of Section 4 of \cite{OkRoStr_JFAA10}:
for $\Lambda>0$,  let $E_{\Lambda}$ be the span of all eigenfunctions
corresponding to eigenvalues $\lambda$ of $-\Delta$ with $\lambda\le\Lambda$, let $P_{\Lambda}$ be the orthogonal projection onto $E_{\Lambda}$, and
set $d_{\Lambda}$ to be the dimension of $E_{\Lambda}$. For an eigenvalue
$\lambda$ of $-\Delta$, write $E_{\lambda}$ for the eigenspace
of $\lambda$, $P_{\lambda}$ for the orthogonal projection onto $E_{\lambda}$,
and $d_{\lambda}$ for the dimension of $E_{\lambda}$.

\subsection{Trace-type Szeg\"{o} limit theorems}\label{subsec:trace}
\global\long\def\Tr{\operatorname{Tr}}
Fix a measurable map $p:X\times(0,\infty)\to\mathbb{R}$ and
let $p(x,-\Delta)$ be the densely defined operator as in \eqref{eq:psdo}.
We assume, unless otherwise stated, that $p(\cdot,\lambda)$ is continuous
for all eigenvalues $\lambda$ of $-\Delta$ and that there is a continuous
map $q:X\mapsto\mathbb{R}$ such that the following limit exists and
is uniform in $x$: 
\begin{equation}
\lim_{\lambda\in\operatorname{sp}(-\Delta),\lambda\to\infty}p(x,\lambda)=q(x)\label{eq:limit_un}.
\end{equation}
In this section we  study the asymptotic behavior of $\Tr F(P_{\Lambda}p(x,-\Delta)P_{\Lambda})$
as $\Lambda\to\infty$ for continuous functions $F$.  Our results are clearly true if $\Vert p\Vert_{\infty}=0$;  henceforth
we assume that $\Vert p\Vert_{\infty}>0$.

\begin{lem}
\label{lem:A_B}The eigenvalues of $P_{\Lambda}p(x,-\Delta)P_{\Lambda}$
are contained in a bounded interval $[A,B]$ for all $\Lambda>0$. 
\end{lem}
\begin{proof}
Let $\varepsilon>0$. Then there is some $\overline{\Lambda}$ such that
if $\lambda\in\operatorname{sp}(-\Delta)$ and $\lambda>\overline{\Lambda},$
then $q(x)-\varepsilon<p(x,\lambda)<q(x)+\varepsilon$. Since 
\[
\langle P_{\lambda}p(x,-\Delta)P_{\lambda}\varphi,\varphi\rangle=\int p(x,\lambda)\varphi(x)^{2}d\mu(x)
\]
for all $\varphi\in E_{\lambda}$, it follows that, as operators,
$P_{\lambda}[q-\varepsilon]P_{\lambda}\le P_{\lambda}p(x,-\Delta)P_{\lambda}\le P_{\lambda}[q+\varepsilon]P_{\lambda}$.
Now $[q+\varepsilon]$ and $[q-\varepsilon]$ are bounded
on $L^{2}(\mu)$, $P_{\overline{\Lambda}}L^{2}(\mu)$ is finite dimensional,
and 
\[
P_{\Lambda}L^{2}(\mu)=\oplus_{\lambda\le\Lambda}P_{\lambda}L^{2}(\mu)=\Bigl( \oplus_{\lambda\le\overline{\Lambda}}P_{\lambda}L^{2}(\mu)\Bigr)\oplus\Bigl(\oplus_{\overline{\Lambda}<\lambda\le\Lambda}P_{\lambda}L^{2}(\mu)\Bigr),  
\]
so the assertion of the lemma holds with $A$ the minimum of the smallest eigenvalue of $[q-\varepsilon]$ and the
smallest eigenvalue of
$P_{\overline{\Lambda}}p(x,-\Delta)P_{\overline{\Lambda}}$, 
and $B$ the maximum of the $L^2$ norm of $[q+\varepsilon]$  and
the largest eigenvalue of $P_{\overline{\Lambda}}p(x,-\Delta)P_{\overline{\Lambda}}$.
\end{proof}

\begin{lem}
\label{lem:posfunctional}Let $\Lambda>0$. Let $p:X\times(0,\infty)\mathbb{\to}\mathbb{R}$
be a bounded measurable function such that $p(\cdot,\lambda_{n})$
is continuous for all $n\in\mathbb{N}$. Then the map on $C[A,B]$ defined by 
\[
F\mapsto\frac{1}{d_{\Lambda}}\Tr F(P_{\Lambda}p(x,-\Delta)P_{\Lambda})
\]
is a continuous non-negative functional, where $A$ and $B$ are as in
Lemma \ref{lem:A_B}.\end{lem}
\begin{proof}
The map is clearly linear and non-negative. Continuity follows
immediately from the fact that $P_{\Lambda}L^{2}(X,\mu)$ is finite
dimensional.
\end{proof}

In preparation for our main result, Theorem \ref{thm:Mainthm}, we prove a version for an increasing sequence
$\{\lambda_{j}\}$ of 6-series eigenvalues where $\lambda_{j}$ has generation of birth $j$.  $E_{j}$ is the eigenspace
corresponding to $\lambda_{j}$ and has  dimension $d_{j}=(3^{j}-3)/2$. For each $1\le N<j$ we have an orthonormal basis
$\{u_{k}\}_{k=1}^{d_{j}}=\{\overline{u}_{k}\}_{k=1}^{d_{j}^{N}}\cup\{v_{k}\}_{k=1}^{\alpha^{N}}$ 
for $E_{j}$ , where the $\overline{u}_{k}$ are localized at scale $N$
and the $v_{k}$ are not localized at scale $N$. Then $d_{j}^{N}=(3^{j}-3^{N+1})/2$,
$\alpha^{N}=(3^{N+1}-3)/2$ and the number of eigenvalues supported
on a single $N$-cell is $m_{j}^{N}=(3^{j-N}-3)/2$. As remarked in
Section 3 of \cite{OkRoStr_JFAA10}, an analogous construction may
be done for a $5$-series eigenfunction with $d_{j}=(3^{j-1}+3)$,
$d_{j}^{N}=(3^{j-1}-3^{N})/2$, $\alpha^{N}=(3^{N}-3)/2$, and
$m_{j}^{N}=(3^{j-N-1}-1)/2$.  It follows that the results that we prove for $6$-series
eigenvalues in Lemma~\ref{lem:simplefunction} and Theorem~\ref{thm:powers_6series} below are also true for the 5-series eigenvalues
with essentially the same proofs. 

Let $P_{j}$ denote the projection
onto $E_{j}$. The matrix $\Gamma_{j}:=P_{j}p(x,-\Delta)P_{j}$ is
a $d_{j}\times d_{j}$ matrix whose entries are given by 
\[
\gamma_{j}(m,l):=\int p(x,\lambda_{j})u_{m}(x)u_{l}(x)d\mu(x).
\]
Our first version of a Szeg\"{o} theorem is for the operator of multiplication by a simple function.
\begin{lem}
\label{lem:simplefunction} Let $N\ge1$ be fixed, and suppose $f=\sum_{i=1}^{3^{N}}a_{i}\chi_{C_{i}}$
is a simple function. Then for all $k\geq0$
\begin{equation}
\lim_{j\to\infty}\frac{\Tr(P_{j}[f]P_{j})^{k}}{d_{j}}=\int f(x){}^{k}d\mu(x).\label{eq:simple-function}
\end{equation}
\end{lem}
\begin{proof}
The case $k=0$ is trivial since $\Tr(I_{E_{j}})=d_{j}$. Let $k>0$
and fix $j>N$. The matrix $P_{j}[f]P_{j}$ has the following structure
with respect to the basis $\{u_{m}\}_{m=1}^{d_{j}}$:
\[
\left[\begin{array}{cc}
R_{j} & 0\\

0 & N_{j}
\end{array}\right],
\]
where $R_{j}$ is a $d_{j}^{N}\times d_{j}^{N}$-matrix corresponding
to the localized eigenvectors and $N_{j}$ is an $\alpha^{N}\times\alpha^{N}$-matrix
corresponding to the non-localized eigenvalues (see also equation
(11) of \cite{OkRoStr_JFAA10} -- notice, however, the small typographical error
in \cite{OkRoStr_JFAA10} where $\star$ should be $0$). Moreover,
the matrix $R_{j}$ consists of $3^{N}$ diagonal blocks; each block
is of the form $a_{i}I_{m_{j}^{N}}$, $i=1,\dots,3^{N}$. We have
that
\[
\Tr(P_{j}[f]P_{j}])^{k}=\Tr(R_{j})^{k}+\Tr(N_{j})^{k}.
\]
We can compute  the first term explicitly:
\[
\Tr(R_{j})^{k}=\sum_{i=1}^{3^{N}}m_{j}^{N}a_{i}^{k}=d_{j}^{N}\sum_{i=1}^{3^{N}}\frac{a_{i}^{k}}{3^{N}}=d_{j}^{N}\int f(x)^{k}d\mu(x).
\]
For the second term we use that each element in $N_{j}$
is smaller in absolute value than $\Vert f\Vert_{\infty}$. Hence
$\vert\Tr(N_{j})^{k}\vert\le(\alpha^{N})^{k}\Vert f\Vert_{\infty}^k$.
Using the fact that $d_{j}-d_{j}^{N}=\alpha^{N}$ we then obtain
that
\[
\left\vert \frac{\Tr(P_{j}[f]P_{j})^{k}}{d_{j}}-\int f(x)^{k}d\mu(x)\right\vert \le\frac{\alpha^{N}}{d_{j}}\int\vert f(x)\vert^{k}d\mu(x)+\frac{(\alpha^{N})^{k}}{d_{j}}\Vert f\Vert_{\infty}.
\]
Since $\lim_{j\to\infty}(\alpha^{N})^{k}/d_{j}=0$
for all $k\ge0$ and $f$ is bounded on the compact
set $X$ the result follows. 
\end{proof}
We use this lemma to prove the following Szeg\"{o}  theorem for pseudodifferential operators for a
sequence of $6$-series eigenvalues, extending \cite[Theorem 1]{OkRoStr_JFAA10}.
\begin{thm}
\label{thm:powers_6series} Let $p:X\times(0,\infty)\mathbb{\to}\mathbb{R}$
be a bounded measurable function such that $p(\cdot,\lambda_{j})$
is continuous for all $j\in\mathbb{N}$. Assume that $\lim_{j\to\infty}p(x,\lambda_{j})=q(x)$
is uniform in $x$. Then
\begin{equation}
\lim_{j\to\infty}\frac{1}{d_{j}}\Tr(P_{j}p(x,-\Delta)P_{j})^{k}=\int_{X}q(x){}^{k}d\mu(x)\label{eq:powsingeeigenv}
\end{equation}
for all $k\ge0$. Consequently,
\begin{equation}
\lim_{j\to\infty}\frac{1}{d_{j}}\Tr F(P_{j}p(x,-\Delta)P_{j})=\int_{X}F(q(x))d\mu(x)\label{eq:funcsingeeigenv}
\end{equation}
for any continuous $F$ supported on $[A,B]$, where 
$A$ and $B$ are as in Lemma~\ref{lem:A_B}.
\end{thm}
\begin{proof}
We prove \eqref{eq:powsingeeigenv} by induction. The case $k=0$ is trivial since
$(P_{j}p(x,-\Delta)P_{j})^{0}=I_{E_{j}}$ and $\Tr(I_{E_{j}})=d_{j}$.
Let $k\ge1$ and assume that \eqref{eq:powsingeeigenv} is true for
all $0\le m<k$. We claim that it suffices to prove the result in
the case when there is $C>0$ such that $p(x,\lambda)\ge C$ for all
$(x,\lambda)\in X\times(0,\infty)$. Indeed, $p(x,\lambda)+2\Vert p\Vert_{\infty}\ge\Vert p\Vert_{\infty}>0$
and, using the induction hypothesis, the result for $p(x,\lambda)$
follows if one proves \eqref{eq:powsingeeigenv} for $p(x,\lambda)+2\Vert p\Vert_{\infty}$.

Let $\varepsilon>0$ and assume that $\varepsilon<C/2$. Since the
function $\lambda\mapsto\lambda^{k}$ is uniformly continuous on $[A,B]$,
there is $0<\delta<\varepsilon$ such that if $\vert\lambda-\lambda^{\prime}\vert<\delta$
then $\vert\lambda^{k}-(\lambda^{\prime})^{k}\vert<\varepsilon$.
Since $q$ is continuous, we can find $N\ge1$ and a simple function
$f_{N}=\sum_{i=1}^{3^{N}}a_{k}\chi_{C_{i}}$, where $\{C_{i}\}_{i=1}^{3^{N}}$
is the decomposition of $X$ into $N$-cells, such that $\Vert q(\cdot)-f_{N}(\cdot)\Vert_{\infty}<\delta/2$.
Therefore, $\vert\int q(x)^{k}d\mu(x)-\int f_{N}(x)^{k}d\mu(x)\vert<\varepsilon$
since $\mu(X)=1$. Moreover, we can find $J\ge1$ such that $\Vert p(\cdot,\lambda_{j})-q(\cdot)\Vert_{\infty}<\delta/2$
for all $j\ge J$. It follows that
\begin{equation}
\frac{C}{2}\le f_{N}(x)-\delta\le p(x,\lambda_{j})\le f_{N}(x)+\delta\label{eq:fnp}
\end{equation}
for all $x\in X$ and $j\ge J$. By increasing $J$, if necessary,
we may assume by Lemma \ref{lem:simplefunction} that
\[
\left\vert \frac{\Tr(P_{j}[f_{N}]P_{j})^{k}}{d_{j}}-\int(f(x_{N}))^{k}d\mu(x)\right\vert <\varepsilon
\]
for all $j\ge J$.

Let $j\ge J$. Equation \eqref{eq:fnp} implies that $0\le P_{j}[f_{N}-\delta]P_{j}\le P_{j}p(x,-\Delta)P_{j}\le P_{j}[f_{N}+\delta]P_{j}$.
We conclude that $\vert\sigma_{m}^{j}-\sigma_{m,N}^{j}\vert<\delta$,
$m=1,\dots,d_{j}$, where $\sigma_{m}^{j}$ are the eigenvalues of
$P_{j}p(x,-\Delta)P_{j}$ and $\sigma_{m,N}^{j}$ are the eigenvalues
of $P_{j}[f_{N}]P_{j}$. Therefore $\vert(\sigma_{m}^{j})^{k}-(\sigma_{m,N}^{j})^{k}\vert<\varepsilon$
for all $m=1,\dots,d_{j}$. Since $\Tr(P_{j}p(x,-\Delta)P_{j})^{k}=\sum_{m}(\sigma_{m}^{j})^{k}$
and $\Tr(P_{j}[f_{N}]P_{j})^{k}=\sum_{m}(\sigma_{m,N}^{j})^{k}$,
we have that
\begin{align*}
\left\vert \frac{\Tr(P_{j}p(x,-\Delta)P_{j})^{k}}{d_{j}}-\int_{X}q(x)^{k}d\mu(x)\right\vert  & \le\left\vert \frac{\Tr(P_{j}p(x,-\Delta)P_{j})^{k}}{d_{j}}-\frac{\Tr(P_{j}[f_{N}]P_{j})^{k}}{d_{j}}\right\vert \\
 & +\left\vert \frac{\Tr(P_{j}[f_{N}]P_{j})^{k}}{d_{j}}-\int_{X}f_{N}(x)^{k}d\mu(x)\right\vert \\
 & +\vert\int_{X}f_{N}(x)^{k}d\mu(x)-\int_{X}q(x)^{k}d\mu(x)\vert\\
 & <\left\vert \frac{\sum_{m=1}^{d_{j}}\bigl((\sigma_{m}^{j})^{k}-(\sigma_{m,N}^{j})^{k}\bigr)}{d_{j}}\right\vert +2\varepsilon\\
 & <3\varepsilon.
\end{align*}
This finishes the proof of \eqref{eq:powsingeeigenv}.

For the last statement of the theorem, a proof similar to that of Lemma \ref{lem:posfunctional} shows
\[
F\mapsto\frac{1}{d_{j}}\Tr F(P_{j}p(x,-\Delta)P_{j})
\]
is a non-negative functional on the set of continuous functions supported
on $[A,B]$ for all $j\ge 1$. The Stone-Weierstrass
theorem implies the result.
\end{proof}

\begin{rem}
The multiplication operator $[f]$ is the special case $p(x,\lambda)=f(x)$ for a continuous function  $f:X\to\mathbb{R}$.  Then for $F$ continuous with support on $[-\Vert[f]\Vert,\Vert[f]\Vert]$
\[
\lim_{j\to\infty}\frac{1}{d_{j}}\Tr F(P_{j}[f]P_{j})=\int_{X}F(f(x))d\mu(x).
\]
Spectral multipliers occur in the special case where $p:(0,\infty)\to\mathbb{R}_{+}$ is a bounded measurable function
such that $\lim_{n\to\infty}p(\lambda_{n})=q$ for some constant $q$.  In this situation if $F$ is continuous with support on $[-\Vert p(-\Delta)\Vert,\Vert p(-\Delta)\Vert]$ then
\[
\lim_{j\to\infty}\frac{1}{d_{j}}\Tr F(P_{j}p(-\Delta)P_{j})=F(q).
\]
\end{rem}

The main goal of this section is to prove the following Szeg\"{o}-type
theorem for pseudo-differential operators on the Sierpi\'nski gasket.
It is an analogue of classical results like \cite[Theorem
29.1.7]{Hor_LPDEIV}, which seem to have originated in \cite{Wid_JFA79}, see
also \cite{Gui_AnnMathSt79}.
\begin{thm}
\label{thm:Mainthm}Let $p:X\times(0,\infty)\mathbb{\to}\mathbb{R}$
be a bounded measurable function such that $p(\cdot,\lambda_{n})$
is continuous for all $n\in\mathbb{N}$. Assume that % $p(x,-\Delta)$
% extends to a bounded operator on $L^{2}(\mu)$ and that
$\lim_{n\to\infty}p(x,\lambda_{n})=q(x)$
is uniform in $x$. Then, for any continuous function $F$ supported
on $[A,B]$, we have that
\begin{equation}
\lim_{\Lambda\to\infty}\frac{1}{d_{\Lambda}}\Tr F(P_{\Lambda}p(x,-\Delta)P_{\Lambda})=\int_{X}F(q(x))d\mu(x), \label{eq:GeneralizedSzego}
\end{equation}
where the interval $[A,B]$ is chosen as in Lemma~\ref{lem:A_B}.
\end{thm}

\begin{rem}\ \\
\begin{enumerate}
\item[(a)] 
If, in addition to the hypotheses of Theorem \ref{thm:Mainthm}, $p$ is a $0$-symbol in the sense
of  \cite[Definition 9.1]{IoRoStr_PSDO_2011}, then for any continuous $F$ on $[-\Vert p(x,-\Delta)\Vert,\Vert p(x,-\Delta)\Vert]$
we have
\[
\lim_{\Lambda\to\infty}\frac{1}{d_{\Lambda}}\Tr F(P_{\Lambda}p(x,-\Delta)P_{\Lambda})=\int_{X}F(q(x))d\mu(x).
\]

\item[(b)] Spectral mulitpliers are the special case where $p:(0,\infty)\to\mathbb{R}$ is a bounded measurable map
such that $\lim_{j\to\infty}p(\lambda_{j})=q$. Then for any continuous $F$ supported on $[-\Vert p(-\Delta)\Vert,\Vert p(-\Delta)\Vert]$
we have
\[
\lim_{\Lambda\to\infty}\frac{1}{d_{\Lambda}}\Tr F(P_{\Lambda}p(-\Delta)P_{\Lambda})=F(q).
\]

\item[(c)] Multiplication by a continuous $f:X\to\mathbb{R}$ is the special case $p(x,\lambda)=f(x)$. Then for any continuous $F$ supported on $[-\Vert[f]\Vert,\Vert[f]\Vert]$
we have
\[
\lim_{\Lambda\to\infty}\frac{1}{d_{\Lambda}}\Tr F(P_{\Lambda}[f]P_{\Lambda})=\int_{X}F(f(x))d\mu(x).
\]
If $\min f(x)>0$ then the above formula  can be obtained
from \cite[Theorem 3]{OkRoStr_JFAA10}. Specifically, the authors
of \cite{OkRoStr_JFAA10} proved that if $F>0$ is continuous, then
\[
\lim_{j\to\infty}\frac{\sum_{m=1}^{d_{j}}F(\sigma_{m}^{(j)})}{d_{j}}=\int F(f(x))d\mu(x),
\]
where $\sigma_{m}^{(j)}$ are the eigenvalues of $P_{j}[f]P_{j}$.
One can obtain our result by taking $F(x)=x^{k}$ and using the fact
that $\Tr(P_{j}[f]P_{j})^{k}=\sum_{m=1}^{d_{j}}(\sigma_{m}^{(j)})^{k}$. 
\end{enumerate}
\end{rem}

\begin{proof}
[Proof of Theorem \ref{thm:Mainthm} ] We prove first that
\begin{equation}
\lim_{\Lambda\to\infty}\frac{1}{d_{\Lambda}}\Tr(P_{\Lambda}p(x,-\Delta)P_{\Lambda})^{k}=\int_{X}q(x)^{k}d\mu(x)\label{eq:Szego_poly}
\end{equation}
for all $k\ge0$. It is easy to prove the formula for  $k=0$.
Let $k\ge1$ and let $\varepsilon>0$. Clearly 
\[
\Tr(P_{\Lambda}p(x,-\Delta)P_{\Lambda})^{k}=\sum_{\lambda\le\Lambda}\Tr(P_{\lambda}p(x,-\Delta)P_{\lambda})^{k}
\]
and $d_{\Lambda}=\sum_{\lambda\le\Lambda}d_{\lambda}$. The proof
of Theorem~\ref{thm:powers_6series} and its equivalent for $5$-series
imply that there is $J> 1$ such that if $\lambda$ is a $6$-series
or a $5$-series eigenvalue with generation of birth at least $J$,
then 
\begin{equation}\label{eqn:56seriesTrapprox}
\left\vert \frac{\Tr(P_{\lambda}p(x,-\Delta)P_{\lambda})^{k}}{d_{\lambda}}-\int q(x)^{k}d\mu(x)\right\vert <\varepsilon.
\end{equation}
\global\long\def\genbirth{\#}
We will write in this proof $\genbirth\lambda$ for the generation
of birth of the eigenvalue $\lambda$. For each $\Lambda>0$ we let
$\Gamma_J(\Lambda)$ be the set of eigenvalues $\lambda\le \Lambda$
such that $\genbirth\lambda> J$ and $\tilde{\Gamma}_J(\Lambda)$ be
the set of eigenvalues $\lambda\le \Lambda$ such that
$\genbirth\lambda\le J$. Notice that $\Gamma_J(\Lambda)$ consists only
of $5$- and $6$-series eigenvalues.

Now fix $\Lambda_{1}>0$ such that $\Gamma_J(\Lambda_1)\ne \emptyset$
and $\tilde{\Gamma}_J(\Lambda_1)\ne \emptyset$. Then, for
all $\Lambda>\Lambda_{1}$,
\begin{align}
\lefteqn{ \left\vert \frac{\Tr(P_{\Lambda}p(x,-\Delta)P_{\Lambda})^{k}}{d_{\Lambda}}-\int q(x)^{k}d\mu(x)\right\vert}\quad& \notag\\
& \le\frac{\sum_{\lambda\in  \tilde{\Gamma}_J(\Lambda) }\vert\Tr(P_{\lambda}p(x,-\Delta)P_{\lambda})^{k}-d_{\lambda}\int q(x)^{k}d\mu(x)\vert}{d_{\Lambda}} \notag\\
& \quad + \frac{\sum_{\lambda\in  \Gamma_J(\Lambda)}\vert\Tr(P_{\lambda}p(x,-\Delta)P_{\lambda})^{k}-d_{\lambda}\int q(x)^{k}d\mu(x)\vert}{d_{\Lambda}}\notag\\
&=I+II.\label{eq:first_step}
\end{align}

The proof of Theorem 2 of \cite{OkRoStr_JFAA10} implies that 
$\lim_{\Lambda\to\infty}\sum_{\lambda\in
  \tilde{\Gamma}_J(\Lambda)}d_{\lambda}/d_{\Lambda}=0$ (see
inequality (22) and the one following it from \cite{OkRoStr_JFAA10}).
Hence, there is $\Lambda_{2}>\Lambda_{1}$ such that, if $\Lambda>\Lambda_{2}$,
we have that 
\[
\frac{\sum_{\lambda\in \tilde{\Gamma}_J(\Lambda)}d_{\lambda}}{d_{\Lambda}}<\frac{\varepsilon}{\Vert p\Vert_{\infty}^{k}+\Vert q\Vert_{\infty}^{k}}.
\]
Since $\Tr(P_{\lambda}p(x,-\Delta)P_{\lambda})^{k}\le d_{\lambda}\Vert p\Vert_{\infty}^{k}$
for all $\lambda$ and $\mu(X)=1$ we obtain for $\Lambda>\Lambda_{2}$
\begin{equation*}
I
\le(\Vert p\Vert_{\infty}^{k}+\Vert
q\Vert_{\infty}^{k})\frac{\sum_{\lambda\in \tilde{\Gamma}_J(\Lambda)}d_{\lambda}}{d_{\Lambda}}<\varepsilon.
\end{equation*}
Finally, by~\eqref{eqn:56seriesTrapprox}, for $\Lambda>\Lambda_{2}$:
\begin{equation*}
II <\varepsilon\frac{\sum_{\lambda\in\Gamma_J(\Lambda)}d_{\lambda}}{d_{\Lambda}}  <\varepsilon.
\end{equation*}
Substitution into~\eqref{eq:first_step} gives
\[
\left\vert \frac{\Tr(P_{\Lambda}p(x,-\Delta)P_{\Lambda})^{k}}{d_{\Lambda}}-\int q(x)^{k}d\mu(x)\right\vert <2\varepsilon.
\]
and an application of the Stone-Weierstrass theorem completes the proof.
\end{proof}

We conclude this subsection  by showing that the hypothesis of uniform convergence of $\lim_{j\to\infty}p(x,\lambda_{j})=q(x)$ in Theorem \ref{thm:powers_6series} may be relaxed if we assume some smoothness conditions on $p$.  The result gives convergence of~\eqref{eq:funcsingeeigenv} along a subsequence of $\{\lambda_{j}\}_{j\in\mathbb{N}}$.
\begin{prop}
\label{prop:subsequence_AA}Let $\{\lambda_{j}\}_{j\in\mathbb{N}}$
be an increasing sequence of $6$-series eigenvalues such that $\lambda_{j}$
has generation of birth $j$, for all $j\ge1$. Assume that $p(\cdot,\lambda_{j})\in\operatorname{Dom}(\Delta)$
for all $j\in\mathbb{N}$. Assume that 
\[
\lim_{j\to\infty}p(x,\lambda_{j})=q(x)\quad\text{for all }x\in X
\]
and that both $p(\cdot,\lambda_{j})$ and $\Delta_{x}p(\cdot,\lambda_{j})$
are bounded uniformly in $j$. Then there is a subsequence $\{\lambda_{k_{j}}\}$
of $\{\lambda_{j}\}$ such that 
\[
\lim_{j\to\infty}\frac{1}{d_{k_{j}}}F(P_{k_{j}}p(x,-\Delta)P_{k_{j}})=\int_{X}F(q(x))d\mu(x)
\]
for all continuous functions $F$ supported on $[A,B]$.
\end{prop}
\begin{proof}
Recall from the proof of \cite[Lemma 3.3]{IR_2010} that there is
a constant $C^{\prime}>0$ such that for any $f\in\operatorname{dom}(\Delta)$
and for any $x,y\in X$, 
\[
\vert f(x)-f(y)\vert\le C^{\prime}R(x,y)\left(\sup_{z\in X}\vert\Delta f(z)\vert+\max_{p,q\in\partial X}\vert f(p)-f(q)\vert\right).
\]
Since, by our hypothesis, $p(x,\lambda_{j})$ and $\Delta_{x}p(x,\lambda_{j})$
are bounded uniformly in $j$, using the above estimate, we can find
a constant $C>0$ such that
\[
\vert p(x,\lambda_{j})-p(y,\lambda_{j})\vert\le CR(x,y)
\]
for all $j\in\mathbb{N}$. Hence the sequence$\{p(\cdot,\lambda_{j})\}$
is equicontinuous. Since the sequence is also uniformly bounded, the
Arzel\'a-Ascoli theorem implies that there is a subsequence $\{p(\cdot,\lambda_{k_{j}})\}$
such that 
\[
\lim_{j\to\infty}p(x,\lambda_{k_{j}})=q(x)
\]
uniformly in $x$. Therefore we can apply Theorem \ref{thm:powers_6series}
to the subsequence $\{p(x,\lambda_{k_{j}})\}$ to obtain the conclusion.
\end{proof}

\subsection{Determinant  Szeg\"{o}-type limit theorems}
Proposition~\ref{thm:Szego_uniflimit} below is a generalization of \cite[Theorem 1]{OkRoStr_JFAA10}, and can be proved using Theorem~\ref{thm:powers_6series}. Specifically,
the result would follow if we set $F(x)=\log(x)$ and consider a positive
map $p$. We  provide an alternative proof  following
the steps of the proof of \cite[Theorem 1]{OkRoStr_JFAA10}. We
hope that our argument will shed light on some of the  technical
details that are sketched in the above mentioned proof. We still assume
that $\{\lambda_{j}\}_{j\ge1}$ is a fixed increasing sequence of
6-series or $5$-series eigenvalues such that $\lambda_{j}$ has generation
of birth $j$ for all $j\in\mathbb{N}$. 
\begin{prop}
\label{thm:Szego_uniflimit}Assume that $p:X\times(0,\infty)\to\mathbb{R}_{+}$
is a measurable function such that $p(\cdot,\lambda_{j})$ is continuous
for all $j\in\mathbb{N}$ and such that there is $C>0$ so that $p(x,\lambda_{j})\ge C$
for all $(x,\lambda_{j})$. Assume that 
\begin{equation}
\lim_{j\to\infty}p(x,\lambda_{j})=q(x)\label{eq:limit}
\end{equation}
exists and the limit is uniform in $x$. Then
\begin{equation}
\lim_{j\to\infty}\frac{1}{d_{j}}\log\det P_{j}p(x,-\Delta)P_{j}=\int_{X}\log q(x)d\mu(x).\label{eq:Szego}
\end{equation}
\end{prop}
\begin{proof}
First we notice that since the limit \eqref{eq:limit} is uniform
in $x$, the function $q$ is continuous on $X$ and $q(x)\ge C>0$
for all $x\in X$. Hence $\log q(x)$ is integrable. 

Let $\varepsilon>0$ and assume that $\varepsilon<\min(1,C/4)$. There
exists $N\ge1$ and a simple function $f_{N}=\sum_{i=1}^{3^{N}}a_{i}\chi_{C_{i}}$,
where $\{C_{i}\}_{i=1}^{3^{N}}$ is a decomposition of $X$ into cells
of order $N$ as in Section~\ref{ssec:analysisonSG}, such that $\Vert q(\cdot)-f_{N}(\cdot)\Vert_{\infty}<(1/2)\min\bigl(\varepsilon,\varepsilon C/2\bigr)$.
Therefore $f_{N}(x)\ge 3C/4>0$ for all $x\in X$. Since $\log$ is
a continuous function, by increasing $N$, if necessary, we can also
assume that 
\[
\vert\int\log q(x)d\mu(x)-\int_{X}\log f_{N}(x)d\mu(x)\vert<\varepsilon.
\]
Since we assume that the limit in \eqref{eq:limit} is uniform, there
is $J\ge1$ such that $\frac{\alpha^{N}}{d_{j}}\log\Vert f_{N}\Vert<\varepsilon$
and $\Vert p(x,\lambda_{j})-q(x)\Vert_{\infty}<(1/2)\min\bigl(\varepsilon,\varepsilon C/2\bigr)$ 
for all $j\ge J$. Hence $\Vert p(\cdot,\lambda_{j})-f_{N}(\cdot)\Vert_{\infty}<\min(\varepsilon,\varepsilon C/2)$
and 
\begin{equation}
1-\varepsilon<\frac{p(x,\lambda_{j})}{f_{N}(x)}<1+\varepsilon\label{eq:poverf}
\end{equation}
 for all $x\in X$ and $j\ge J$.

Let $j\ge J$ be fixed. Recall that the operator $P_{j}p(x,-\Delta)P_{j}$
has a block structure 
\[
\Gamma_{j}=\begin{bmatrix}R_{j} & *\\
* & N_{j}
\end{bmatrix}
\]
with respect to the basis $\{u_{k}\}_{k=1}^{d_{j}}$. The entries
of $\Gamma_{j}$ are given by
\[
\gamma_{j}(i,k)=\int_{X}p(x,\lambda_{j})u_{i}(x)u_{k}(x)d\mu(x).
\]
Similarly, the operators $P_{j}[f_{N}(1-\varepsilon)]P_{j}$ and $P_{j}[f_{N}(1+\varepsilon)]P_{j}$
have block structures (see also \cite{OkRoStr_JFAA10}) 
\[
\Gamma_{j,N}(\pm\varepsilon)=\begin{bmatrix}R_{j,N}(\pm\varepsilon) & 0\\
0 & N_{j,N}(\pm\varepsilon)
\end{bmatrix}
\]
respectively. The blocks $R_{j}$, $R_{j,N}(-\varepsilon)$, and $R_{j,N}(\varepsilon)$
are $d_{j}^{N}\times d_{j}^{N}$ blocks corresponding to the ``localized''
part, while $N_{j}$, $N_{j,N}(-\varepsilon)$, and $N_{j,N}(\varepsilon)$
correspond to the ``nonlocalized'' part. The inequality \eqref{eq:poverf}
implies that 
\[
0\le\langle\Gamma_{j,N}(-\varepsilon)g,g\rangle\le\langle\Gamma_{j}g,g\rangle\le\langle\Gamma_{j,N}(\varepsilon)g,g\rangle
\]
 for all $g\in E_{j}$. Thus, as operators, $0\le\Gamma_{j,N}(-\varepsilon)\le\Gamma_{j}\le\Gamma_{j,N}(\varepsilon)$.
 \cite[Corollary 7.7.4]{HoJo_MatAn90} implies that
\[
\det\Gamma_{j,N}(-\varepsilon)\le\det\Gamma_{j}\le\det\Gamma_{j,N}(\varepsilon).
\]
Hence 
\[
\log\det\Gamma_{j,N}(-\varepsilon)\le\log\det\Gamma_{j}\le\log\det\Gamma_{j,N}(\varepsilon).
\]
The block $R_{j,N}(\varepsilon)$ consists of $3^{N}$ blocks of the
form $a_{i}(1+\varepsilon)I_{m_{j}^{N}}$, $i=1,\dots,3^{N}$. Hence
\[
\log\det\Gamma_{j,N}(\varepsilon)=d_{j}^{N}\int\log f_{N}(x)d\mu(x)+d_{j}^{N}\log(1+\epsilon)+\log\det N_{j,N}(\varepsilon).
\]
An estimate as in \cite{OkRoStr_JFAA10} shows that $\vert\log\det N_{j,N}(\varepsilon)\vert\le\alpha^{N}\log\Vert f_{N}\Vert_{\infty}+\alpha^{N}\log(1+\varepsilon)$.
Similarly,
\[
\log \det\Gamma_{j,N}(-\varepsilon)=d_{j}^{N}\int\log f_{N}(x)d\mu(x)+d_{j}^{N}\log(1-\epsilon)+\log\det N_{j,N}(-\varepsilon)
\]
and $\vert\log\det N_{j,N}(\varepsilon)\vert\le\alpha^{N}\log\Vert f_{N}\Vert_{\infty}+\alpha^{N}\log(1-\varepsilon)$.
Using the fact that $d_{j}=d_{j}^{N}+\alpha^{N}$, we obtain that
\begin{align*}
\left\vert \frac{\log\det P_{j}p(x,-\Delta)P_{j}}{d_{j}}-\int\log q(x)d\mu(x)\right\vert  & \le\vert\int\log q(x)d\mu(x)-\int\log f_{N}(x)d\mu(x)\vert\\
 & \quad+\frac{\alpha^{N}}{d_{j}}(\Vert\log f_{N}\Vert_{1}+\log\Vert f_{N}\Vert_{\infty})+\varepsilon.
\end{align*}
This implies the conclusion.
\end{proof}

The proof of Proposition \ref{prop:subsequence_AA} can be easily
adapted to the proof of the following corollary.
\begin{cor}
Assume that $p:X\times(0,\infty)\to\mathbb{R}_{+}$ is a measurable
function such that $p(\cdot,\lambda_{j})$ is continuous for all $j\in\mathbb{N}$
and such that there is $C>0$ so that $p(x,\lambda_{j})\ge C$ for
all $(x,\lambda_{j})$. Moreover, assume that $p(\cdot,\lambda)\in\operatorname{Dom}(\Delta)$
for all $\lambda\in \operatorname{sp}(-\Delta)$ and that 
\[
\lim_{j\to\infty}p(x,\lambda_{j})=q(x)\quad\text{for all }x\in X
\]
 and that both $p(x,\lambda)$ and $\Delta_{x}p(x,\lambda)$ are bounded
on $X\times \operatorname{sp}(-\Delta)$. Then there is a subsequence $\{\lambda_{k_{j}}\}$
of $\{\lambda_{j}\}$ such that 
\[
\lim_{j\to\infty}\frac{1}{d_{k_{j}}}\log\det P_{k_{j}}p(x,-\Delta)P_{k_{j}}=\int_{X}\log q(x)d\mu(x).
\]
\end{cor}
More generally, we have: 
\begin{prop}
\label{thm:GeneralSzego} Assume that $p:X\times(0,\infty)\to\mathbb{R}_{+}$
is a measurable function such that $p(\cdot,\lambda)$ is continuous
for all $\lambda\in \operatorname{sp}(-\Delta)$ and such that there is $C>0$ so that
$p(x,\lambda)\ge C$ for all $(x,\lambda)\in X\times\operatorname{sp}(-\Delta)$. Assume
that 
\begin{equation}
\lim_{\lambda\to\infty}p(x,\lambda)=q(x)\label{eq:limit-1}
\end{equation}
exists and the limit is uniform in $x$. Then
\begin{equation}
\lim_{\Lambda\to\infty}\frac{1}{d_{\Lambda}}\log\det P_{\Lambda}p(x,-\Delta)P_{\Lambda}=\int_{X}\log q(x)d\mu(x).\label{eq:Szego-1}
\end{equation}
\end{prop}
\begin{proof}
The proof is  similar to the proof of Theorem \ref{thm:Mainthm}
and we omit the details.
\end{proof}

\section{Examples and Applications}\label{sec:ExamplesApp}
In this section we consider some applications of our main results. 
Our examples show how one can extract asymptotic behaviour of
operators from elementary integrals, but we also give a result on
spectral clustering for generalized Schr\"{o}dinger operators
following a line of argument due to  Widom \cite{Wid_JFA79}. 
  
\subsection{Examples}
\label{subsec:Example}
 Our results can be used to obtain asymptotics of  Riesz and Bessel operators, as well as
some specific generalized Schr\"{o}dinger operators.
\begin{example}
Let $p:(0,\infty)\to\mathbb{R}$ be $p(\lambda)=1+\lambda^{-\beta}$,
where $\beta$ is a positive real number. Then $p(-\Delta)=I+(-\Delta)^{-\beta}$.
We know from \cite{IR_2010} that $p(-\Delta)$ is bounded on $L^{p}(\mu)$
for all $p>1$. Let $F$ be a continuous function supported on $(-\Vert p(-\Delta)\Vert,\Vert p(-\Delta)\Vert)$.
Since $\lim_{\lambda\to\infty}p(\lambda)=1$, we have that
\[
\lim_{j\to\infty}\frac{\Tr F(P_{j}p(-\Delta)P_{j})}{d_{j}}=F(1)
\]
for any increasing sequence $\{\lambda_{j}\}$ of $6$-series or $5$-series
eigenvalues such that $\lambda_{j}$ has generation of birth $j$.
Moreover
\[
\lim_{\Lambda\to\infty}\frac{\Tr F(P_{\Lambda}p(-\Delta)P_{\Lambda})}{d_{\Lambda}}=F(1).
\]
Since $p(\lambda)\ge 1$ for all $\lambda>0$, we can apply Proposition
\ref{thm:Szego_uniflimit} and Proposition \ref{thm:GeneralSzego} and   obtain that
\[
\lim_{j\to\infty}\frac{1}{d_{j}}\log\det P_{j}p(-\Delta)P_{j}=0
\]
and
\[
\lim_{\Lambda\to\infty}\frac{1}{d_{\Lambda}}\log\det P_{\Lambda}p(-\Delta)P_{\Lambda}=0.
\]

Let $p:(0,\infty)\to\mathbb{R}$ be $p(\lambda)=1+(1+\lambda)^{-\beta}$,
where $\beta$ is a positive real number. Then the corresponding operator is  the Bessel operator given by $p(-\Delta)=I+(I-\Delta)^{-\beta}$. All the conclusions of the previous examples hold for this operator as well. 
\end{example}

%\begin{example}
%Let $p:(0,\infty)\to\mathbb{R}$ be $p(\lambda)=1+(1+\lambda)^{-\alpha}$,
%where $\alpha$ is a positive real number. Then $p(-\Delta)=I+(I-\Delta)^{-\alpha}$.
%We know from \cite{IR_2010} that $p(-\Delta)$ is bounded on $L^{p}(\mu)$
%for all $p>1$. Let $F$ be a continuous function supported on $(-\Vert p(-\Delta)\Vert,\Vert p(-\Delta)\Vert)$.
%Since $\lim_{\lambda\to\infty}p(\lambda)=1$, we have that
%\[
%\lim_{j\to\infty}\frac{\Tr F(P_{j}p(-\Delta)P_{j})}{d_{j}}=F(1)
%\]
%for any increasing sequence $\{\lambda_{j}\}$ of $6$-series or $5$-series
%eigenvalues such that $\lambda_{j}$ has generation of birth $j$.
%Moreover
%\[
%\lim_{\Lambda\to\infty}\frac{\Tr F(P_{\Lambda}p(-\Delta)P_{\Lambda})}{d_{\Lambda}}=F(1).
%\]
%Since $p(\lambda)\ge 1$ for all $\lambda>0$, we can apply Proposition
%\ref{thm:Szego_uniflimit} and Proposition \ref{thm:GeneralSzego} and   obtain that
%\[
%\lim_{j\to\infty}\frac{1}{d_{j}}\log\det P_{j}p(-\Delta)P_{j}=0
%\]
%and
%\[
%\lim_{\Lambda\to\infty}\frac{1}{d_{\Lambda}}\log\det P_{\Lambda}p(-\Delta)P_{\Lambda}=0.
%\]
%
%\end{example}

\begin{example}
Let $p(x,\lambda)=q(\lambda)+\chi(x)$, where $q$ is a bounded measurable
function on $(0,\infty)$ such that $\lim_{\lambda\to\infty}q(\lambda)=l$
exists and $\chi$ is continuous on $X$. Then $p(x,-\Delta)$
is a generalized Schr\"{o}dinger operator.  Since it is the sum of two bounded
operators it is bounded on $L^{2}(\mu)$. Let $F$ be a continuous function supported on $(-\Vert p(x,-\Delta)\Vert,\Vert p(x,-\Delta)\Vert)$.
Then, if $\{\lambda_{j}\}_{j\ge1}$ is an increasing sequence of $ $6-series
or, respectively, $5$-series eigenvalues, we have both of the
following equalities:
\[
\lim_{j\to\infty}\frac{\Tr F(P_{j}p(x,-\Delta)P_{j})}{d_{j}}=\int F(l+\chi(x))d\mu(x)=\lim_{\Lambda\to\infty}\frac{\Tr F(P_{\Lambda}p(x,-\Delta)P_{\Lambda})}{d_{\Lambda}}.
\]
% and 
% \[
% \lim_{\Lambda\to\infty}\frac{\Tr F(P_{\Lambda}p(x,-\Delta)P_{\Lambda})}{d_{\Lambda}}=\int F(l+\chi(x))d\mu(x).
% \]
In particular, if $l=0$, we have that
\[
\lim_{j\to\infty}\frac{\Tr F(P_{j}p(x,-\Delta)P_{j})}{d_{j}}=\int F(\chi(x))d\mu(x)=\lim_{\Lambda\to\infty}\frac{\Tr F(P_{\Lambda}p(x,-\Delta)P_{\Lambda})}{d_{\Lambda}}.
\]
% and 
% \[
% \lim_{\Lambda\to\infty}\frac{\Tr F(P_{\Lambda}p(x,-\Delta)P_{\Lambda})}{d_{\Lambda}}=\int F(\chi(x))d\mu(x).
% \]

If we further assume $p(x,\lambda)\ge C>0$ for all
$(x,\lambda)\in X\times(0,\infty)$ then
\[
\lim_{j\to\infty}\frac{1}{d_{j}}\log\det P_{j}p(x,-\Delta)P_{j}=\int\log(l+\chi(x))d\mu(x)%=\lim_{\Lambda\to\infty}\frac{1}{d_{\Lambda}}\log\det P_{\Lambda}p(x,-\Delta)P_{\Lambda}.
\]
% and
% \[
% \lim_{\Lambda\to\infty}\frac{1}{d_{\Lambda}}\log\det P_{\Lambda}p(x,-\Delta)P_{\Lambda}=\int\log(l+\chi(x))d\mu(x).
% \]
and the same is true if we replace $P_{j}$ with $P_{\Lambda}$ and
$d_{j}$ with $d_{\Lambda}$, so in particular if $l=0$, 
\[
\lim_{j\to\infty}\frac{1}{d_{j}}\log\det P_{j}p(x,-\Delta)P_{j}=\int\log(\chi(x))d\mu(x)=\lim_{\Lambda\to\infty}\frac{1}{d_{\Lambda}}\log\det P_{\Lambda}p(x,-\Delta)P_{\Lambda}.
\]
% and
% \[
% \lim_{\Lambda\to\infty}\frac{1}{d_{\Lambda}}\log\det P_{\Lambda}p(x,-\Delta)P_{\Lambda}=\int\log(\chi(x))d\mu(x).
% \]

\end{example}

\subsection{Application: Asymptotics of eigenvalue clusters
for general Schr\"{o}dinger operators}\label{subsec:asymptotics-of-eigenvalue}

Let $p:(0,\infty)\to\mathbb{R}$ be a measurable function and let
$\chi$ be a real-valued bounded measurable function on $X$. We call
the operator $H=p(-\Delta)+[\chi]$ a \emph{generalized Sch}r\emph{\"{o}dinger
operator} with  \emph{potential} $\chi$.  We study
the asymptotic behavior of spectra of generalized Schr\"{o}dinger
operators with continuous potentials and continuous $p$, generalizing
some results of \cite{OkStr_PAMS07}. 

We begin with a lemma that is a generalization of the key  \cite[Lemma 1]{OkStr_PAMS07}.
\begin{lem}
\label{lem:minmaxlemma}Let $p:(0,\infty)\to\mathbb{R}$ be a continuous
function such that there is $A\in\mathbb{R}$ with $p(\lambda)\ge A$
for all $\lambda\ge\lambda_{1}$, where $\lambda_{1}$ is the smallest
positive eigenvalue of $-\Delta$. For $i=1,2$, let $\chi_{i}$ be
real-valued bounded measurable functions on $X$. Let $H_{i}=p(-\Delta)+[\chi_{i}]$
denote the corresponding generalized Schr\"{o}dinger operators. For
$n\ge1$, the $n$th eigenvalues $\nu_{n}^{i}$ of $H_{i}$, $i=1,2$,
satisfy the following inequality:
\[
\vert\nu_{n}^{1}-\nu_{n}^{2}\vert\le\Vert\chi_{1}-\chi_{2}\Vert_{L^{\infty}}.
\]
\end{lem}
\begin{proof}
The hypothesis implies that
\[
\langle H_{i}f,f\rangle\ge(A+\min_{X}\chi_{i})\Vert f\Vert_{2}^{2}
\]
for all $f\in D$, $i=1,2$. Hence $H_{i}$ is bounded from below,
$i=1,2$. The remainder of the proof is identical to the proof of \cite[Lemma 1]{OkStr_PAMS07}.
 % Theorem XIII.2 of \cite{ReSi_MMMPIV} implies that
% \begin{equation}
% \nu_{n}^{i}=\sup_{\phi_{1},\dots,\phi_{n-1}}\inf_{\psi\in[\phi_{1},\dots,\phi_{n}]^{\perp};\Vert\psi\Vert_{L^{2}}=1}\langle H_{i}\psi,\psi\rangle.\label{eq:ntheigenvalue}
% \end{equation}
% Moreover, we have that
% \begin{align*}
% \langle H_{1}\psi,\psi\rangle & =\langle H_{2}\psi,\psi\rangle+\langle(\chi_{1}-\chi_{2})\psi,\psi\rangle\\
%  & \le\langle H_{2}\psi,\psi\rangle+\Vert\chi_{1}-\chi_{2}\Vert_{L^{\infty}}\Vert\psi\Vert_{L^{2}}.
% \end{align*}
% This inequality together with \eqref{eq:ntheigenvalue} implies the
% result.
\end{proof}

Assume that $p:(0,\infty)\to\mathbb{R}$
is a continuous function, that there is $\overline{\lambda}>0$
such that $p$ is increasing on $[\overline{\lambda},\infty)$ and
\begin{equation}
\vert p(\lambda)-p(\lambda^{\prime})\vert\ge c\vert\lambda-\lambda^{\prime}\vert^\beta\label{eq:pcondition}
\end{equation}
 for all $\lambda,\lambda^{\prime}\ge\overline{\lambda}$ and some constants $c>0$ and $\beta>0$.
Let $\chi$ be a continuous function on $X$ and $H=p(-\Delta)+[\chi]$
be the corresponding Schr\"{o}dinger operator. Let
$\{\lambda_{j}\}$ be a sequence of $6$-series eigenvalues of $-\Delta$
such that the separation between $\lambda_{j}$
and the next higher and lower eigenvalues of $-\Delta$ grows exponentially
in $j$. For example, if $\lambda_{1}$ is any 6-series eigenvalue with
generation of birth $1$ and if $\lambda_{j}=5^{j-1}\lambda_{1}$,
then the sequence $\{\lambda_{j}\}_{j\ge1}$ satisfies our assumption
(\cite{OkStr_PAMS07}, \cite[Chapter 3]{Str_Prin06}). Let $\tilde{\Lambda}_{j}$ be the portion of the spectrum
of $H$ lying in $[p(\lambda_{j})+\min\chi,p(\lambda_{j})+\max\chi]$.
Lemma \ref{lem:minmaxlemma} implies that, for large $j$, $\tilde{\Lambda}_{j}$
contains exactly $d_{j}$ eigenvalues $\{\nu_{i}^{j}\}_{i=1}^{d_{j}}$.
We call this the $p(\lambda_{j})$ \emph{cluster} of the eigenvalues
of $H$. If we translate by $p(\lambda_j)$ units to the left we obtain
clusters that lie in a fixed interval.  Define the characteristic
measure of the $p(\lambda_{j})$ cluster 
of $H$ by
\[
\Psi_{j}(\lambda)=\frac{1}{d_{j}}\sum_{i=1}^{d_{j}}\delta(\lambda-(\nu_{i}^{j}-p(\lambda_{j})).
\]
and observe that integrating the function $x^k$ against this measure yields
\begin{equation}
  \label{eq:psi_j}
  \langle\Psi_j,x^k\rangle=\frac{1}{d_j}\Tr \bigl(\overline{P}_j
  (p(-\Delta)+[\chi]-p(\lambda_j))\overline{P}_j\bigr)^k, 
\end{equation} 
for all $k\ge 0$, where $\overline{P}_j$ is the spectral projection for
$p(-\Delta)+[\chi]$ associated with the $p(\lambda_j)$
cluster. This allows us to analyze the weak limit of $\Psi_j$ using
Theorem \ref{thm:Mainthm}.

\begin{thm}\label{thm:specclusters}
The sequence $\{\Psi_{j}\}_{j\ge1}$ converges weakly to the pullback
of the measure $\mu$ under $\chi$ defined for all continuous
functions $f$ supported on $[\min \chi, \max \chi]$ by
\[
\langle\Psi_{0},f\rangle=\int_{X}f(\chi(x))d\mu(x).
\]
\end{thm}
% \begin{proof}
% The proof is very similar to that
% of Theorem 1 of \cite{OkStr_PAMS07} so we omit the details.
% \end{proof}

% The proof of the theorem follows from the trace-type Szeg\"{o} limit
% theorem together with the following two lemmas. 
%  In particular, $p$ satisfies the hypothesis of both
% Corollary \ref{cor:spectrumifpinfty} and Lemma \ref{lem:minmaxlemma}.

\begin{lem}{\label{lem:pow_clust}}
Assume that $N>0$ and that $\chi_N=\sum_{i=1}^N a_i \chi_{C_i}$ is a
simple function,  where $\{C_{i}\}_{i=1}^{3^{N}}$ is a partition
of $X$ into $N$-cells. Let $H_{N}=p(-\Delta)+[\chi_{N}]$
be the corresponding generalized Schr\"{o}dinger operator,
$\tilde{\Lambda}_{j}^{N}$ the $p(\lambda_j)$ cluster of $H_N$, and let
$\overline{P}_{j}^{N}$ be the spectral projection for $H_N$ associated
with the $p(\lambda_j)$ cluster. Then
%\begin{equation}
%  \label{eq:lemm_powclust}
\[
  \lim_{j\to \infty}\frac{\Tr \bigl(\overline{P}_j^N
    (p(-\Delta)+[\chi_N]-p(\lambda_j))\overline{P}_j^N\bigr)^k}{d_j}=\lim_{j\to
    \infty}\frac{(P_j[\chi_N]P_j)^k}{d_j}=\int_X \chi_N(x)^kd\mu(x),
  \]
%\end{equation}
  for all $k\ge 0$.
\end{lem}
\begin{proof}
   Consider
  the first equality. 
If $\overline{u}$ is an eigenfunction of the basis
of $E_{j}$ that is localized in a single $N$-cell $C_{i}$, then
$H_{N}\overline{u}=(p(\lambda_{j})+a_{i})\overline{u}$. Thus
$\lambda_{i}^{j}:=p(\lambda_{j})+a_{i}$ is an eigenvalue of $H_{N}$
with multiplicity at least $m_{j}^{N}$. 
Doing this identifies $d_{j}^{N}$ eigenvalues in $\tilde{\Lambda}_{j}^{N}$
and we let $A_{j}^{N}$ be the set of the remaining $\alpha^{N}$
eigenvalues. Hence 
\[
\Tr(\overline{P}_j^N
    (p(-\Delta)+[\chi_N]-p(\lambda_j))\overline{P}_j^N\bigr)^k=m_{j}^{N}\sum_{i=1}^{3^N}a_i^k+\sum_{\nu\in
    A_{j}^{N}} \nu^k.
\]
Recall from the proof of Lemma \ref{lem:simplefunction} that
$\Tr(P_j[\chi_N]P_j)^k=m_{j}^{N}\sum_{i=1}^{3^N}a_i^k +\Tr (N_j)^k$. Since both $\sum_{\nu\in
    A_{j}^{N}} \nu^k$ and $\Tr (N_j)^k$ are bounded by a constant
  times $(\alpha^N)^k$,  the equality of the two limits follows.
The second equality is from Theorem \ref{thm:Mainthm}.
\end{proof}

%Then $A_{j}^{N}$ % might contain eigenvalues $\lambda_{i}^{j}$ if their
% multiplicity is greater then $m_{j}^{N}$ and the number of eigenvalues
% in $A_{j}^{N}$ equals $\alpha^{N}$.

\begin{proof}[Proof of Theorem \ref{thm:specclusters}]
Since $\chi$ is continuous, it can be approximated uniformly by a
sequence of simple functions $\chi_{N}=\sum_{i=1}^{3^{N}}a_{i}\chi_{C_{i}}$, of the type previously described.
%where $\{C_{i}\}_{i=1}^{3^{N}}$ is the partition of $X$ into $N$-cells.
We can choose $\chi_{N}$
such that $\min\chi\le\min\chi_{N}$ and $\max\chi_{N}\le\max\chi$
for all $N$ from which  $\tilde{\Lambda}_{j}^{N}$ is contained
in $[p(\lambda_{j})+\min\chi,p(\lambda_{j})+\max\chi]$. Moreover there
is $J_1$ depending only $\Vert \chi\Vert_\infty$ such that $j\ge J_1$
implies  $\tilde{\Lambda}_{j}^{N}$  contains
$d_{j}$ eigenvalues $\{\tilde{\lambda}_{i}^{j}\}_{i=1}^{d_{j}}$.

From Lemma \ref{lem:minmaxlemma}, for all $i$
and $j$,
$
\vert\nu_{i}^{j}-\tilde{\lambda}_{i}^{j}\vert\le\Vert\chi-\chi_{N}\Vert_{\infty}
$. Therefore, when $j\ge J_1$,
\begin{align*}
&\Bigl|\Tr \bigl(\overline{P}_j
    (p(-\Delta)+[\chi]-p(\lambda_j))\overline{P}_j\bigr)^k-
\Tr \bigl(\overline{P}_j^N
(p(-\Delta)+[\chi_N]-p(\lambda_j))\overline{P}_j^N\bigr)^k\Bigr|\\
&\quad
=\sum_{i=1}^{d_j}\bigl\vert(\nu_i^{j}-p(\lambda_j))^k-(\tilde{\lambda}_{i}^{j}-p(\lambda_j))^k\bigr\vert
\le d_j k\Vert \chi\Vert_{\infty}^{k-1}\Vert \chi-\chi_N\Vert_{\infty}.
  \end{align*}
Then, with $j\ge J_1$,
\begin{align*}
  \biggl\vert \langle \Psi_j,x^k\rangle-\int_{X} \chi^kd\mu\biggr\vert
  &\le k\Vert \chi\Vert_{\infty}^{k-1}\Vert \chi-\chi_N\Vert_{\infty}
  +\int_X\vert \chi_N^{k}-\chi^{k}\vert d\mu\\
 &\quad  +d_{j}^{-1}\biggl\vert \Tr \bigl(\overline{P}_j^N
(p(-\Delta)+[\chi_N]-p(\lambda_j))\overline{P}_j^N\bigr)^k - \int_{X}
\chi_N^{k}d\mu\biggr\vert.
\end{align*}
For sufficiently large $N$ both $\Vert\chi-\chi_N\Vert_\infty$ and
$\Vert \chi-\chi_N\Vert_{1}$ are less than $\varepsilon$ and we can
take $j>J_1$ so large that the last term is smaller than $\varepsilon$
by Lemma  \ref{lem:pow_clust}.
Applying the Stone-Weierstrass theorem completes the proof.
\end{proof}

\bibliographystyle{amsplain}
\bibliography{Szegobibl}

\def\noopsort#1{}\def\cprime{$'$} \def\sp{^}
\providecommand{\bysame}{\leavevmode\hbox to3em{\hrulefill}\thinspace}
\providecommand{\MR}{\relax\ifhmode\unskip\space\fi MR }
\providecommand{\MRhref}[2]{%
  \href{http://www.ams.org/mathscinet-getitem?mr=#1}{#2}
}
\providecommand{\href}[2]{#2}
\begin{thebibliography}{10}

\bibitem{BaKi_JLMS97}
Martin~T. Barlow and Jun Kigami, \emph{Localized eigenfunctions of the
  {L}aplacian on p.c.f.\ self-similar sets}, J. London Math. Soc. (2)
  \textbf{56} (1997), no.~2, 320--332.

\bibitem{DalStrVin_FDESG_JFAA99}
Kyallee Dalrymple, Robert~S. Strichartz, and Jade~P. Vinson, \emph{Fractal
  differential equations on the {S}ierpinski gasket}, J. Fourier Anal. Appl.
  \textbf{5} (1999), no.~2-3, 203--284.

\bibitem{FukShi_92}
M.~Fukushima and T.~Shima, \emph{On a spectral analysis for the {S}ierpi\'nski
  gasket}, Potential Anal. \textbf{1} (1992), no.~1, 1--35.

\bibitem{FOT}
Masatoshi Fukushima, Yoichi Oshima, and Masayoshi Takeda, \emph{Dirichlet forms
  and symmetric {M}arkov processes}, extended ed., de Gruyter Studies in
  Mathematics, vol.~19, Walter de Gruyter \& Co., Berlin, 2011. \MR{2778606
  (2011k:60249)}

\bibitem{Gui_AnnMathSt79}
Victor Guillemin, \emph{Some classical theorems in spectral theory revisited},
  Seminar on {S}ingularities of {S}olutions of {L}inear {P}artial
  {D}ifferential {E}quations ({I}nst. {A}dv. {S}tudy, {P}rinceton, {N}.{J}.,
  1977/78), Ann. of Math. Stud., vol.~91, Princeton Univ. Press, Princeton,
  N.J., 1979, pp.~219--259. \MR{547021 (81b:58045)}

\bibitem{Hor_LPDEIV}
Lars H{\"o}rmander, \emph{The analysis of linear partial differential
  operators. {IV}}, Classics in Mathematics, Springer-Verlag, Berlin, 2009,
  Fourier integral operators, Reprint of the 1994 edition.

\bibitem{HoJo_MatAn90}
Roger~A. Horn and Charles~R. Johnson, \emph{Matrix analysis}, Cambridge
  University Press, Cambridge, 1990, Corrected reprint of the 1985 original.

\bibitem{IR_2010}
Marius {Ionescu} and Luke {Rogers}, \emph{{Complex Powers of the Laplacian on
  Affine Nested Fractals as Calder\'on-Zygmund operators}}, Commun. Pure Appl.
  Anal., in press (2013).

\bibitem{IoRoStr_PSDO_2011}
Marius Ionescu, Luke Rogers, and Robert Strichartz, \emph{Pseudo-differential
  operators on fractals and other metric measure spaces}, Rev. Mat. Iberoam.
  \textbf{29} (2013), no.~4, 1159--1190.

\bibitem{Kig_CUP01}
Jun Kigami, \emph{Analysis on fractals}, Cambridge Tracts in Mathematics, vol.
  143, Cambridge University Press, Cambridge, 2001.

\bibitem{OkRoStr_JFAA10}
Kasso~A. Okoudjou, Luke~G. Rogers, and Robert~S. Strichartz, \emph{Szeg\"o
  limit theorems on the {S}ierpi\'nski gasket}, J. Fourier Anal. Appl.
  \textbf{16} (2010), no.~3, 434--447.

\bibitem{OkStr_PAMS07}
Kasso~A. Okoudjou and Robert~S. Strichartz, \emph{Asymptotics of eigenvalue
  clusters for {S}chr\"odinger operators on the {S}ierpi\'nski gasket}, Proc.
  Amer. Math. Soc. \textbf{135} (2007), no.~8, 2453--2459 (electronic).

\bibitem{rammal1983random}
Rammal Rammal and G{\'e}rard Toulouse, \emph{Random walks on fractal structures
  and percolation clusters}, Journal de Physique Lettres \textbf{44} (1983),
  no.~1, 13--22.

\bibitem{Simon}
Barry Simon, \emph{Szeg{\H o}'s theorem and its descendants}, M. B. Porter
  Lectures, Princeton University Press, Princeton, NJ, 2011, Spectral theory
  for $L{^{2}}$ perturbations of orthogonal polynomials.

\bibitem{Str_Prin06}
Robert~S. Strichartz, \emph{Differential equations on fractals: A tutorial},
  Princeton University Press, Princeton, NJ, 2006.

\bibitem{Tep_JFA98}
Alexander Teplyaev, \emph{Spectral analysis on infinite {S}ierpi\'nski
  gaskets}, J. Funct. Anal. \textbf{159} (1998), no.~2, 537--567.

\bibitem{Wid_JFA79}
Harold Widom, \emph{Eigenvalue distribution theorems for certain homogeneous
  spaces}, J. Funct. Anal. \textbf{32} (1979), no.~2, 139--147. \MR{534671
  (80h:58054)}

\end{thebibliography}

\end{document}